\documentclass[11pt]{article}
\usepackage{amsmath,amsfonts,amssymb,amsthm,mathrsfs,mathtools,stmaryrd}
\usepackage[a4paper,vmargin={3.5cm,3.5cm},hmargin={2.5cm,2.5cm}]{geometry}
\usepackage[font=sf, labelfont={sf,bf}, margin=1cm]{caption}
\usepackage{xcolor,framed}
\usepackage{ulem}
\usepackage[applemac]{inputenc} 
\usepackage{ae,aecompl}
\usepackage[bookmarksopen=false,colorlinks=true,urlcolor=blue,citecolor=blue,citebordercolor=blue,linkcolor=blue]{hyperref}
\usepackage[english]{babel}

\renewcommand{\P}{\mathbb{P}}
\newcommand{\E}{\mathbb{E}}
\newcommand{\R}{\mathbb{R}}
\def\N{\mathbb{N}}
\def\e{{\rm e}}
\def\d{{\rm d}}
\def\z{\mathbf{Z}}
\def \1{1 \mkern -6mu 1} 

\newtheorem{theorem}{Theorem}[section]
\newtheorem{definition}[theorem]{Definition}
\newtheorem{proposition}[theorem]{Proposition}
\newtheorem{lemma}[theorem]{Lemma}
\newtheorem{corollary}[theorem]{Corollary}
\theoremstyle{definition}
\newtheorem{remark}[theorem]{Remark}  

\title{\vspace{-2cm}\textsc{On a two-parameter Yule-Simon distribution} \vspace{0.5cm}}

\author{Erich Baur \thanks{Bern University of Applied Sciences \hfill
    \texttt{erich.baur@bfh.ch}} \qquad Jean Bertoin \thanks{Institut f\"ur
    Mathematik, Universit\"at Z\"urich.\hfill
    \texttt{jean.bertoin@math.uzh.ch}} }
\date{}
\DeclareSymbolFont{extraup}{U}{zavm}{v}{n}
\DeclareMathSymbol{\varheart}{\mathalpha}{extraup}{86}
\DeclareMathSymbol{\vardiamond}{\mathalpha}{extraup}{87}
\makeatletter
\renewcommand*{\@fnsymbol}[1]{\ensuremath{\ifcase#1\or \spadesuit \or \varheart \or \vardiamond \or \clubsuit\or \or
   \mathsection\or \mathparagraph\or \|\or **\or \dagger\dagger
   \or \ddagger\ddagger \else\@ctrerr\fi}}
\makeatother

\begin{document}
\normalem
\maketitle

\begin{abstract}  
  We extend the classical one-parameter Yule-Simon law to a
  version depending on two parameters, which in part appeared in \cite{Be}
  in the context of a preferential attachment algorithm with fading
  memory. By making the link to a general branching process with
  age-dependent reproduction rate, we study the tail-asymptotic behavior of
  the two-parameter Yule-Simon law, as it was already initiated
  in \cite{Be}. Finally, by superposing mutations to the branching process,
  we propose a model which leads to the full two-parameter range of the
  Yule-Simon law, generalizing thereby the work of Simon \cite{Sim} on
  limiting word frequencies.
 \end{abstract}

\medskip


\medskip


\section{Introduction}

The standard Yule process $Y=(Y(t))_{t\geq 0}$ is a basic population model
in continuous time and with values in $\N \coloneqq \{1, 2, \ldots\}$. It
describes the evolution of the size of a population started from a single
ancestor, where individuals are immortal and give birth to children at unit
rate, independently one from the other. It is well-known that for every
$t\geq 0$, $Y(t)$ has the geometric distribution with parameter
$\e^{-t}$. As a consequence, if $T_{\rho}$ denotes an exponentially
distributed random time with parameter $\rho >0$ which is independent of
the Yule process, then for every $k\in\N$, there is the identity
\begin{equation}\label{E:YSd}
  \P(Y(T_{\rho})=k)=\rho \int_0^{\infty} \e^{-\rho t} (1-\e^{-t})^{k-1} \e^{-t} \d t = \rho B(k,\rho+1) , 
\end{equation}
 where $B$ is the beta function.
 
 The discrete distribution in \eqref{E:YSd} has been introduced by
 H.A. Simon \cite{Sim} in 1955 and is nowadays referred to as the
 Yule-Simon distribution with parameter $\rho$. It arises naturally in
 preferential attachment models and often explains the occurrence of heavy
 tail variables in stochastic modeling. Indeed, the basic estimate
$$B(k,\rho+1) \sim  \Gamma(\rho+1) k^{-(\rho +1)} \qquad \text{as }k\to \infty,$$
implies that the Yule-Simon distribution has a fat tail with exponent
$\rho$.

The present work is devoted to a two-parameter generalization of the
Yule-Simon distribution, which results from letting the fertility (i.e. the
reproduction rate) of individuals in the population model depend on their
age. Specifically, imagine that now the rate at which an individual of age
$a\geq 0$ begets children is $\e^{-\theta a}$ for some fixed $\theta\in\R$. So
for $\theta >0$ the fertility decays with constant rate $\theta$ as
individuals get older, whereas for $\theta<0$, the fertility increases with
constant rate $-\theta$. Denote the size of the population at time $t$ by
$Y_{\theta}(t)$. In other words, $Y_{\theta}=(Y_{\theta}(t))_{t\geq 0}$ is
a general (or Crump-Mode-Jagers) branching process, such that the point
process on $[0,\infty)$ that describes the ages at which a typical
individual begets a child is Poisson with intensity measure
$\e^{-\theta t}\d t$. For $\theta=0$, $Y_0=Y$ is the usual Yule process.

\begin{definition} \label{D1} Let $\theta \in\R$ and $\rho >0$. Consider
  $Y_{\theta}$ as above and let $T_{\rho}$ be an exponential random time
  with parameter $\rho >0$, independent of $Y_{\theta}$.  We call the law
  of the discrete random variable
$$X_{\theta, \rho}\coloneqq Y_{\theta}(T_{\rho})$$ 
the {\rm Yule-Simon distribution with parameters $(\theta, \rho)$}.
\end{definition}

A key difference with the original Yule-Simon distribution, which
corresponds to $\theta =0$, is that no close expression for the
two-parameter distribution is known\footnote{ Although the probability
  $\P(X_{\theta, \rho}=1)$ can easily be computed in terms of an incomplete
  Gamma function, the calculations needed to determine
  $\P(X_{\theta, \rho}=k)$ for $k\geq 2$ become soon
  intractable.}. Actually, the general branching process $Y_{\theta}$ is
not even Markovian for $\theta \neq 0$, and its one-dimensional
distributions are not explicit.  This generalization of the Yule-Simon
distribution has recently appeared in \cite{Be} for $\theta >0$ and
$\rho >1$, in connection with a preferential attachment model with fading
memory in the vein of Simon's original model. We shall point out in
Section \ref{sec:PopModel} that the range of parameters $\theta \leq 0$ and
$\rho >0$ arises similarly for a family of related models.

One of the purposes of the present contribution is to describe some
features of the two-parameter Yule-Simon law, notably by completing
\cite{Be} and determining the tail-asymptotic behavior of
$X_{\theta, \rho}$. It was observed in \cite{Be} that the parameter
$\theta =1$ is critical, in the sense that when $\theta <1$,
$X_{\theta, \rho}$ has a fat tail with exponent $\rho/(1-\theta )$, whereas
when $\theta >1$, some exponential moments of positive order of
$X_{\theta, \rho}$ are finite. We show here in Section \ref{sec:Tails} that
when $\theta >1$, the tail of $X_{\theta, \rho}$ is actually decaying
exponentially fast with exponent $\ln \theta -1+1/\theta$.  Further, in the
critical case $\theta=1$, we show that $X_{1,\rho}$ has a stretched
exponential tail with stretching exponent $1/3$.

By superposing independent neutral mutations at each birth with fixed
probability $1-p\in(0,1)$ to the classical Yule process, the original
Yule-Simon law with parameter $\rho=1/p$ captures the limit number of
species of a genetic type chosen uniformly at random among all types, upon
letting time tend to infinity. This fact is essentially a rephrasing of
Simon's results in \cite{Sim}. We give some (historical) background in
Section \ref{sec:PopModel} and extend Simon's observations to more general
branching processes, for which the two-parameter distribution from
Definition \ref{D1} is observed.

In a similar vein, the number of species belonging to a genus chosen
uniformly at random has been studied for generalized Yule models in several
works by Lansky, Polito, Sacerdote and further co-authors, both at fixed
times $t$ and upon letting $t\rightarrow\infty$. For instance,
in \cite{LaPoSa1}, the linear birth process governing the growth of species
is replaced by a birth-and-death process, whereas in \cite{LaPoSa2}, a
fractional nonlinear birth process is considered instead. Both works are
formulated in the framework of World Wide Web modeling. Recently,
Polito \cite{Po} changed also the dynamics how different genera appear,
leading to a considerably different limit behavior.

The rest of this article is organized as follows. In the following
Section \ref{sec:Ytheta}, we analyze the branching process $Y_\theta$
introduced above and study its large-time behavior. In
Section \ref{sec:PoissonRep}, we develop an integral representation for the
tail distribution of the two-parameter Yule-Simon law, which lies at the
heart of our study of the tail asymptotics of $X_{\theta,\rho}$ in the
subsequent Section \ref{sec:Tails}. This part complements the
work \cite{Be} and contains our main results. In the last
Section \ref{sec:PopModel}, we relate the generalized Yule-Simon
distribution to a population model with mutations, in the spirit of Simon's
original work \cite{Sim}.

\section{Preliminaries on the general branching process $Y_{\theta}$}
\label{sec:Ytheta} The purpose of this section is to gather some basic
features about the general branching process $Y_{\theta}$ that has been
described in the introduction. We start with a construction of $Y_{\theta}$
in terms of a certain branching random walk.

Specifically, we consider a sequence $\z=(\z_n)_{n\geq 0}$ of point
processes on $[0,\infty)$ which is constructed recursively as
follows. First, $\z_0=\delta_0$ is the Dirac point mass at $0$, and for any
$n\geq 0$, $\z_{n+1}$ is obtained from $\z_n$ by replacing each and every
atom of $\z_n$, say located at $z\geq 0$, by a random cloud of atoms
$\{z+\omega^z_i\}_{i=1}^{{\mathrm N}^z}$, where
$\{\omega^z_i\}_{i=1}^{{\mathrm N}^z}$ is the family of atoms of a Poisson
point measure on $[0,\infty)$ with intensity $\e^{-\theta t} \d t$ and to
different atoms $z$ correspond independent such Poisson point measures. In
particular, each ${\mathrm N}^z$ has the Poisson distribution with
parameter $1/\theta$ when $\theta >0$, whereas ${\mathrm N}^z=\infty$
a.s. when $\theta \leq 0$.  If we now interpret $[0,\infty)$ as a set of
times, the locations of atoms as birth-times of individuals, and consider
the number of individuals born on the time-interval $[0,t]$,
$$Y_{\theta}(t)\coloneqq \sum_{n=0}^{\infty} \z_n([0,t])\,,\qquad t\geq 0,$$
then $Y_{\theta} =(Y_{\theta}(t))_{t\geq 0}$ is a version of the general
branching process generalizing the standard Yule process that was discussed
in the introduction.

We readily observe the following formula for the first
moments:

\begin{proposition}\label{P01} 
One has for every $t\geq 0$:
$$\E(Y_{\theta}(t))=  \left\{
\begin{matrix} (\e^{(1-\theta )t}- \theta)/ (1-\theta) & \text{ if }\theta\neq 1,\\
1+t & \text{ if } \theta =1.
\end{matrix} \right. $$
\end{proposition}
\begin{proof} By definition, the intensity of the point process $\z_1$ is
  $\e^{-\theta t}\d t$, and by the branching property, the intensity of
  $\z_n$ is the $n$-th convolution product of the latter. Considering
  Laplace transforms, we see that for any $q>1-\theta$:
\begin{align*}
  q \int_0^{\infty}\E\left( Y_{\theta}(t)\right) \e^{-qt} \d t &= q \int_0^{\infty} \e^{-qt} \left(\sum_{n=0}^{\infty} \E\left(\z_n([0,t])\right)\right) \d t \\
  &= \sum_{n=0}^{\infty}  \E\left(  \int_0^{\infty} \e^{-qt} \z_n(\d t) \right) \\
  &=  \sum_{n=0}^{\infty} (\theta+q)^{-n}\\
  &= \frac{\theta +q}{\theta+q-1}.
\end{align*}
Inverting this Laplace transform yields our claim. 
\end{proof}
\begin{remark}\label{R00} The calculation above shows that 
  a two-parameter Yule-Simon variable $X_{\theta, \rho}$, as in Definition
  \ref{D1}, is integrable if and only if $\theta + \rho >1$, and in that
  case we have
$$\E(X_{\theta, \rho})=\frac{\theta +\rho}{\theta +\rho-1}.$$
\end{remark} 

Proposition \ref{P01} ensures the finiteness of the branching process
$Y_{\theta}$ observed at any time. Further, it should be plain that the
atoms of the branching random walk $\z$ (at all generations) occupy
different locations. Thus $Y_{\theta}$ is a counting process, in the sense
that its sample paths take values in $\N$, are non-decreasing and all its
jumps have unit size.  We next discuss its large time asymptotic behavior,
and in this direction, we write
$$Y_{\theta}(\infty)= \lim_{t\to \infty}\uparrow Y_{\theta}(t)\in \bar
\N\coloneqq \N \cup\{\infty\}$$ for its terminal value.

\begin{proposition}\label{P02} \begin{itemize}
  \item[(i)] If $\theta >0$, then $Y_{\theta}(\infty)$ has the Borel
    distribution with parameter $1/\theta$, viz.
$$ \P(Y_{\theta}(\infty)=n)= \frac{\e^{-n/\theta}(n/\theta)^{n-1}}{n!} \qquad \text{for every }n\in \N.$$
In particular $\P(Y_{\theta}(\infty)<\infty)=1$ if and only if $\theta\geq 1$. 
\item[(ii)] If $\theta <1$, then 
$$\lim_{t\to \infty} \e^{(\theta-1)t} Y_{\theta}(t) = W_{\theta} \qquad \text{in probability,}$$
where $W_{\theta}\geq 0$ is a random variable in $L^k(\P)$ for any
$k\geq 1$.  Moreover, the events $\{W_{\theta}=0\}$ and
$\{Y_{\theta}(\infty)<\infty\}$ coincide a.s., and are both negligible
(i.e. have probability $0$) if $\theta\leq 0$.
\end{itemize} 
\end{proposition}
\begin{proof} (i) When $\theta > 0$, $(\z_n([0,\infty))_{n\geq 0}$ is a
  Galton-Watson process with reproduction law given by the Poisson
  distribution with parameter $1/\theta$.  In particular, it is critical
  for $\theta=1$, sub-critical for $\theta>1$, and super-critical for
  $\theta<1$.  In this setting, $Y_{\theta}(\infty)$ is the total
  population generated by a single ancestor in this Galton-Watson process;
  since the reproduction law is Poisson, it is well-known that
  $Y_{\theta}(\infty)$ is distributed according to the Borel distribution
  with parameter $1/\theta$.

  (ii) The claims follow by specializing to our setting well-known results
  on general branching processes. More precisely, the fact that
  $\int_0^{\infty} \e^{-(1-\theta)t} \e^{-\theta t} \d t=1$ shows that the
  so-called Malthus exponent of the general branching process $Y_{\theta}$
  equals $1-\theta$. Then we just combine Theorem A of Doney
  \cite{Doney72}, Theorem 1 of Bingham and Doney \cite{BingDon}, and
  Theorem 3.1 in Nerman \cite{Ner}.
\end{proof}

Finally, it will be convenient to also introduce 
$$F_{\theta}(t)\coloneqq \sum_{n=0}^{\infty} \int_0^t \e^{-\theta(t-s)} \z_n(\d s), \qquad t\geq 0.$$
We call $F_{\theta}=(F_{\theta}(t))_{t\geq 0}$ the fertility process; it
can be interpreted as follows. Recall that an atom, say at $s\geq 0$, of
the branching random walk (at any generation $n$) is viewed as the
birth-time of an individual, and $t-s$ is thus its age at time $t\geq
s$. The times at which this individual begets children form a Poisson point
measure on $[s,\infty)$ with intensity $\e^{-\theta (t-s)} \d t$. Hence,
$F_{\theta}(t)$ should be viewed as the total rate of birth (therefore the
name fertility) at time $t$ for the population model described by
$Y_{\theta}$.

\begin{proposition}\label{P03} The fertility process $F_{\theta}$ is a
  Markov process on $(0,\infty)$ with infinitesimal generator
\begin{equation}\label{E:genF}
{\mathcal G}_{\theta}f(x)=-\theta  x f'(x)+x(f(x+1)-f(x)), 
\end{equation}
say for $f:(0,\infty)\to \R$ a bounded ${\mathcal C}^1$ function with
bounded derivative $f'$.
\end{proposition} 
\begin{remark} \label{R01} Specialists will have recognized from
  \eqref{E:genF} that the fertility $F_{\theta}$ is a so-called continuous
  state branching process; see \cite{Jirina} and Chapter 12 in \cite{Kyp}
  for background.
\end{remark}

\begin{proof} The fertility process starts from $F_{\theta}(0)=1$, takes
  values in $(0,\infty)$, decays exponentially with constant rate $\theta$
  (by convention, exponential decay with rate $\theta<0$ means exponential
  increase with rate $-\theta >0$), and makes jumps of unit size
  corresponding to birth events at time $t$. That is, there is the identity
\begin{equation}\label{E:dF}
F_{\theta}(t) = Y_{\theta}(t) -\theta \int_0^t F_{\theta}(s)\d s.
\end{equation}
The claim should now be intuitively obvious since $F_{\theta}(t)$ is also
the rate at time $t$ at which the counting process $Y_{\theta}$ has a jump
of unit size.

To give a rigorous proof, we introduce the filtration
${\mathcal F}_t=\sigma({\mathbf 1}_{[0,t]}\z_n: n\in \N)$ for $t\geq 0$.
Since the point measure $\z_1$ is Poisson with intensity
$\e^{-\theta s}\d s$, the process
$$ \z_1([0,t]) -\int_0^t\e^{-\theta s} \d s, \qquad t\geq 0$$
is an $({\mathcal F}_t)$-martingale. By the branching property, we have
more generally that for any $n\geq 0$,
$$ \z_{n+1}([0,t]) -\int_0^t \int_0^s\e^{-\theta (s-r)} \z_n(\d r)  \d s, \qquad t\geq 0$$
is also an $({\mathcal F}_t)$-martingale, and summing over all generations, we conclude that 
\begin{equation} \label{E:Ycomp}  Y_{\theta}(t)-\int_0^t F_{\theta}(s)\d s \qquad \text{
is an $({\mathcal F}_t)$-martingale. }
\end{equation}
As $Y_{\theta}$ is a counting process, we deduce from \eqref{E:dF} that for
any bounded ${\mathcal C}^1$ function $f:(0,\infty)\to \R$ with bounded
derivative, there is the identity
$$f(F_{\theta}(t))-f(1)= -\theta \int_0^t F_{\theta}(s) f'(F_{\theta}(s)) \d s + \int_0^t(f(F_{\theta}(s-)+1)-f(F_{\theta}(s-))\d Y_{\theta}(s).$$
We now see from \eqref{E:Ycomp} that 
$$f(F_{\theta}(t)) - \int_0^t {\mathcal G}_{\theta}(F_{\theta}(s)) \d s \qquad \text{
is an $({\mathcal F}_t)$-martingale. }
$$
It is readily checked that the martingale problem above is well-posed, and
the statement follows; see Section 4.4 in \cite{EtKu} for background.
\end{proof}

We point out that for $f(x)=x$, we get
${\mathcal G}_{\theta}f=(1-\theta ) f$, and it follows that
$\E(F_{\theta}(t))=\e^{(1-\theta )t}$ for all $t\geq 0$. We then see from
\eqref{E:dF} that for $\theta \neq 1$,
$$\E(Y_{\theta}(t))= \e^{(1-\theta )t} + \theta \int_0^t \e^{(1-\theta )s} \d s= \frac{1}{1-\theta }(\e^{(1-\theta )t}- \theta),
$$
and that  $\E(Y_{1}(t))=1+t$ for $\theta=1$, hence recovering Proposition \ref{P01}. 

\section{Poissonian representation for the tail distribution}
\label{sec:PoissonRep}
The purpose of this section is to point at the following representation of
the tail distribution of the two-parameter Yule-Simon distribution.  We
first introduce a standard Poisson process $N=(N(t))_{t\geq 0}$. We
write $$\gamma(n)\coloneqq \inf\{t>0: N(t)=n\}$$ for every $n\in \N$ (so
that $\gamma(n)$ has the Gamma distribution with parameters $(n,1)$), and
\begin{equation}\label{E:zeta}
\zeta_{\theta}\coloneqq \inf\{t>0: N(t)+1-\theta t =0\}
\end{equation} for $\theta\in \R$ (in particular $\zeta_{\theta}= \infty$ a.s. when $\theta \leq 0$). 

\begin{proposition}\label{P04} Let $\theta \in \R$ and $\rho >0$. For every
  $n\in \N$, one has
$$\P(X_{\theta, \rho}>n) = \E\left( \exp\left( -\rho\int_0^{\gamma(n)} (N(t)+1-\theta t)^{-1} \d t\right) {\mathbf 1}_{\gamma(n)<\zeta_{\theta}}\right).$$
\end{proposition}

This identity could be inferred from \cite{Be}; for the sake of
completeness, we shall provide here an independent proof based on
Proposition \ref{P03} and the identity \eqref{E:dF}.
\begin{proof}[Proof of Proposition \ref{P04}]
  Observe from Proposition \ref{P03} that the infinitesimal generator
  ${\mathcal G}_{\theta}$ of the fertility process fulfills
$$x^{-1} {\mathcal G}_{\theta}f(x)=-\theta  f'(x)+(f(x+1)-f(x)), \qquad x>0,$$
and that the right-hand side is the infinitesimal generator of a standard
Poisson process with drift $-\theta $ absorbed at $0$.  If we write
$$\xi_{\theta}(t)\coloneqq N(t\wedge \zeta_{\theta})+1-\theta (t\wedge \zeta_{\theta}) \qquad \text{for }t\geq 0 , $$
so that the process $\xi_{\theta}$ is that described above and started from
$\xi_{\theta}(0)=1$, then by Volkonskii's formula (see e.g. Formula (21.6)
of Section III.21 in \cite{RogWil}), the fertility can be expressed as a
time-change of $\xi_{\theta}$. Specifically, the map
$t\mapsto \int_0^t \xi_{\theta}(s)^{-1}\d s$ is bijective from
$[0,\zeta_{\theta})$ to $\R_+$, and if we denote its inverse by
$\sigma_{\theta}$, then the processes $F_{\theta}$ and
$\xi_{\theta}\circ \sigma_{\theta}$ have the same distribution; we can
henceforth assume that they are actually identical.

In this setting, we can further identify
$\sigma_{\theta}(t)=\int_0^t F_{\theta}(s)\d s$ and then deduce from
\eqref{E:dF} that $Y_{\theta}(t) =1+N(\sigma_{\theta}(t))$. As a
consequence, if we write
$$\tau_{\theta}(n)\coloneqq \inf\{t>0: Y_{\theta}(t)>n\},$$
 then we have also 
$$
\tau_{\theta}(n)= \inf\{t>0: N\circ \sigma_{\theta}(t)=n\}
=\left\{
\begin{matrix} \int_0^{\gamma(n)}\xi_{\theta}(s)^{-1} \d s &\text{if } \gamma(n)<\zeta_{\theta}, \\
\infty &\text{otherwise.}
\end{matrix}
\right.
$$

Finally, recall from Definition \ref{D1} that $T_{\rho}$ has the
exponential distribution with parameter $\rho >0$ and is independent of
$Y_{\theta}$, so
$$\P(X_{\theta, \rho}>n) = \P(Y_{\theta}(T_{\rho})>n) = \E\left(\exp(-\rho \tau_{\theta}(n)) {\mathbf 1}_{\tau_{\theta}(n)<\infty}\right).$$
This completes the proof. 
\end{proof}

\begin{remark} Following up Remark \ref{R01}, the application of
  Volkonskii's formula in the proof above amounts to the well-known
  Lamperti's transformation that relates continuous state branching
  processes and L\'evy processes without negative jumps via a time-change;
  see \cite{CLUB} for a complete account.
\end{remark}

We conclude this section by pointing at a simple inequality between the
tail distributions of Yule-Simon processes with different parameters. 
\begin{corollary}\label{C3} 
\begin{itemize}
\item[(i)] The random variable $X_{\theta, \rho}$ decreases stochastically
  in the parameters $\theta$ and $\rho$. That is, for every
  $\theta' \geq \theta$ and $\rho'\geq \rho >0$, one has
$$\P(X_{\theta', \rho'}>n) \leq \P(X_{\theta, \rho}>n)\qquad \text{for all }n\in\N.$$
\item[(ii)] For every $\theta\in \R$, $\rho >0$ and $a>1$, one has
$$\P(X_{\theta, \rho}>n)^a \leq \P(X_{\theta, a\rho}>n)\qquad \text{for all }n\in\N.$$
\end{itemize}
\end{corollary}
\begin{proof} (i) It should be plain from the construction of the general
  branching process $Y_{\theta}$ in the preceding section, that for any
  $\theta\leq \theta'$, one can obtain $Y_{\theta'}$ from $Y_{\theta}$ by
  thinning (i.e. random killing of individuals and their descent). In
  particular $Y_{\theta}$ and $Y_{\theta'}$ can be coupled in such a way
  that $Y_{\theta}(t) \geq Y_{\theta'}(t)$ for all $t\geq 0$.  Obviously,
  we may also couple $T_{\rho}$ and $T_{\rho'}$ such that
  $T_{\rho}\geq T_{\rho'}$ (for instance by defining
  $T_{\rho'}=\frac{\rho}{\rho'} T_{\rho}$), and our claim follows from the
  fact that individuals are eternal in the population model. Alternatively,
  we can also deduce the claim by inspecting Proposition \ref{P04}.

(ii) This follows immediately from H\"older's inequality and Proposition \ref{P04}.
\end{proof}

\section{Tail asymptotic  behaviors}
\label{sec:Tails}
We now state the main results of this work which completes that of
\cite{Be}.  The asymptotic behavior of the tail distribution of a two
parameter Yule-Simon distribution exhibits a phase transition between
exponential and power decay for the critical parameter $\theta =1$; here is
the precise statement.

\begin{theorem}\label{T1} Let $\rho>0$.
\begin{enumerate}
\item [(i)] If $\theta   <1$, then  there exists a constant $C=C(\theta,
  \rho)\in(0,\infty)$ such that, as $n\to
  \infty$:
$$\P(X_{\theta, \rho}>n)\sim C n^{-\rho/(1-\theta  )}.
$$
\item [(ii)]   If $\theta >1$, then   as $n\to
  \infty$:
$$\ln \P(X_{\theta, \rho}>n)\sim -(\ln \theta-1+1/\theta)n. 
$$
\end{enumerate}
\end{theorem}
This phase transition can be explained as follows. We rewrite Proposition \ref{P04} in the form
$$\P(X_{\theta, \rho}>n) = \E\left( \exp\left( -\rho\int_0^{\gamma(n)}
    (N(t)+1-\theta t)^{-1} \d t\right) \mid
  {\gamma(n)<\zeta_{\theta}}\right) \times
\P({\gamma(n)<\zeta_{\theta}}).$$ 
On the one hand, the probability that $\gamma(n)<\zeta_{\theta}$ remains
bounded away from $0$ when $\theta<1$ and decays exponentially fast when
$\theta>1$.  On the other hand, for $\theta<1$, the integral
$\int_0^{\gamma(n)} (N(t)+1-\theta t)^{-1} \d t$ is of order $\ln n$ on the
event $\{\gamma(n)<\zeta_{\theta}\}$, and therefore the first term in the
product decays as a power of $n$ when $n$ tends to infinity.  Last, when
$\theta >1$, the first term in the product decays sub-exponentially fast.

In the critical case $\theta=1$, we observe from the combination of Theorem
\ref{T1} and Corollary \ref{C3} that the tail of $X_{1,\rho}$ is neither
fat nor light, in the sense that
$$\exp(-\alpha n) \ll \P(X_{1, \rho}>n)\ll n^{-\beta}$$ for all $\alpha,
\beta >0$. We obtain a more precise estimate of stretched exponential
type. In the following statement, $f\lesssim g$ means
$\limsup_{n\rightarrow\infty}{f(n)}/{g(n)}\leq 1$.

\begin{theorem}\label{T2} Consider the critical case $\theta =1$, and let
  $\rho>0$. Then we have as $n\rightarrow\infty$:
$$-20(\rho^2n)^{1/3}\lesssim \ln \P(X_{1, \rho}>n)\lesssim -(1/2)^{1/3}(\rho^2n)^{1/3}.
$$
\end{theorem}

\begin{remark}
Note that Theorems \ref{T1} and \ref{T2}  entail that the
  series $\sum_{n\geq 0} \P(X_{\theta, \rho}>n)$ converges if and only if
  $\theta + \rho >1$, in agreement with Remark \ref{R00}.
  
  The methods we use to prove Theorem \ref{T2} seem not to be fine enough
  to obtain the exact asymptotics of
  $n^{-1/3}\ln \P(X_{1, \rho}>n)$. More specifically, for 
  the lower bound we employ estimates for first exits through moving
  boundaries proved by Portnoy \cite{Por} first for Brownian motion and
  then transferred via the KMT-embedding to general sums of independent
  random variables.  The constant $c_1=20$ is an (rough) outcome of our proof and
  clearly not optimal.

  For obtaining the upper bound, we consider an appropriate exponential
  martingale and apply optional stopping. The constant $c_2=(1/2)^{1/3}$
  provides the best value given our method, but is very likely not the
  optimal value neither.
\end{remark}

Theorem \ref{T1}(i) has been established in Theorem 1(ii) of \cite{Be} in
the case $\theta\in(0,1)$ and $\rho>1$. Specifically, the parameters
$\alpha$ and $\bar p$ there are such that, in the present notation,
$\theta=\alpha/\bar p(\alpha+1)$ and $\bar p(\alpha+1)=1/\rho$. Taking this
into account, we see that the claim here extends Theorem 1(ii) in \cite{Be}
to a larger set of parameters. The argument is essentially the same,
relying now on Proposition \ref{P02}(ii) here rather than on the less
general Corollary 2 in \cite{Be}, and we won't repeat it.

We next turn our attention to the proof of Theorem \ref{T1}(ii), which
partly relies on the following elementary result on first-passage times of
Poisson processes with drift (we refer to \cite{Do1, Do2, DoRiv} for
related estimates in the setting of general random walks and L\'evy
processes).
\begin{lemma} \label{L1} Let $b >1$, $x> 0$, and define
  $\nu(x)\coloneqq \inf\{t>0: bt-N(t)>x\}$.  The distribution of the
  integer-valued variable $b\nu(x)-x$ fulfills
$$\P(b \nu(x)-x= n)=\frac{1}{n!} \e^{-(x+n)/b}x(x+n)^{n-1} b^{-n}
\sim \frac{ x \e^{x(1-1/b)} }{\sqrt{2\pi n^3}} \e^{n(1-1/b-\ln b)}\qquad
\text{as }n\to \infty.$$ As a consequence,
$$\lim_{t\to \infty} t^{-1}\ln \P(\nu(x)\geq t)= -(b \ln b -b +1).$$ 
\end{lemma}

\begin{proof} The event $b \nu(x)=x$ holds if and only if the Poisson
  process $N$ stays at $0$ up to time $x/b$ at least, which occurs with
  probability $\e^{-x/b}$. The first identity in the statement is thus
  plain for $n=0$.  Next, note that, since the variable $b \nu(x)-x$ must
  take integer values whenever it is finite, there is the identity
$$b \nu(x)-x=\inf\{j\geq 0: N((j+x)/b )=j\}.$$
On the event $N(x/b )=k\in\N$, set $N'(t)=N(t+x/b )-k$, and write 
$$b \nu(x)-x=\inf\{j\in \N: N'(j/b )=j-k\}.$$
Since $N'$ is again a standard Poisson process, Kemperman's formula (see,
e.g. Equation (6.3) in \cite{PiSF}) applied to the random walk
$N'(\cdot/b)$ gives for any $n\geq k$
$$\P(b \nu(x)-x= n\mid N(x/b )=k)= \frac{k}{n}\cdot \frac{\e^{-n/b }(n/b )^{n-k}}{ (n-k)!} .$$
Since $N(x/b)$ has the Poisson distribution with parameter $x/b$, this
yields for any $n\geq 1$
\begin{align*}
  \P(b \nu(x)-x= n) &= \e^{-(x+n)/b}\sum_{k=1}^n  \frac{(x/b)^k}{k!}\cdot \frac{k}{n}\cdot \frac{(n/b )^{n-k}}{ (n-k)!} \\
                    &= \frac{1}{n!} \e^{-(x+n)/b} (x/b)\sum_{k=1}^n  (x/b)^{k-1} (n/b )^{n-k}\cdot \frac{(n-1)!}{(k-1)!(n-k)!}\\
                    &= \frac{1}{n!} \e^{-(x+n)/b}x(x+n)^{n-1} b^{-n},
\end{align*}
where we used Newton's binomial formula at the third line.  The second
assertion in the claim follows from Stirling's formula, and the third one
is a much weaker version.
\end{proof}

 We can now proceed to the proof of  Theorem \ref{T1}(ii). 
\begin{proof}[Proof of Theorem \ref{T1}(ii)] 
The upper-bound is easy. Indeed on the one hand, Proposition \ref{P04} yields
$$\P(X_{\theta, \rho}>n) \leq  \P({\gamma(n)<\zeta_{\theta}}),$$
and on the other hand, since $N(\zeta_{\theta})+1=\theta \zeta_{\theta}$,
on the event $\{\gamma(n)<\zeta_{\theta}\}$, one has obviously
$N(\zeta_{\theta})\geq n$, and {\sl a fortiori} $\theta
\zeta_{\theta}>n$. Thus $\P(X_{\theta, \rho}>n)$ is bounded from above by
$\P(\zeta_{\theta}>n/\theta)$, and we conclude from Lemma \ref{L1}
specialized for $x=1$ and $b=\theta$ that
$$\limsup_{n\to \infty} n^{-1} \ln \P(X_{\theta, \rho}>n) \leq -(\ln \theta - 1+1/\theta ).$$

In order to establish a converse lower bound, let $\varepsilon\in(0,1)$ be
arbitrarily small, and consider the event
$$\Lambda(n,\theta,\varepsilon)\coloneqq\{N(t)+1-\varepsilon -(\theta+\varepsilon)t\geq 0\text{ for all }0\leq t \leq \gamma(n)\}.$$
On that event, one has $N(t)+1-\theta t \geq \varepsilon(1+ t)$ for all
$0\leq t \leq \gamma(n)$, and hence
$$\exp\left( -\rho\int_0^{\gamma(n)} (N(t)+1-\theta t)^{-1} \d t\right)
\geq \left(\gamma(n)+1\right)^{-\rho/\varepsilon} \geq
(n/(\theta+\varepsilon))^{-\rho/\varepsilon},$$ where for the second
inequality, we used that
$N(\gamma(n))+1-\varepsilon \geq (\theta+\varepsilon)\gamma(n)$.

We are left with estimating $\P(\Lambda(n,\theta,\varepsilon))$. Set
$b =\theta+\varepsilon$ and use the notation of Lemma \ref{L1}, so that
$$\Lambda(n,\theta,\varepsilon)=\{\gamma(n)< \nu(1-\varepsilon)\}.$$ 
On the event $\{b \nu(1-\varepsilon)\geq n \}$, one has 
$$N(\nu(1-\varepsilon))=b  \nu(1-\varepsilon)-1+\varepsilon \geq n+\varepsilon-1,$$
so actually 
$\gamma(n)<\nu(1-\varepsilon)$.
Hence 
$\{b \nu(1-\varepsilon)\geq n\}\subset \Lambda(n,\theta,\varepsilon)$, 
and we conclude from Lemma \ref{L1}
that
$$\liminf_{n\to \infty} n^{-1} \ln \P(\Lambda(n,\theta,\varepsilon)) \geq  1-\ln b -1/b .$$

Putting the pieces together, we have shown that for any $b >\theta$
$$\liminf_{n\to \infty} n^{-1} \ln \E\left( \exp\left(
    -\rho\int_0^{\gamma(n)} (N(t)+1-\theta t)^{-1} \d t\right) {\mathbf
    1}_{\gamma(n)<\zeta_{\theta}}\right) \geq 1-\ln b -1/b. $$ Thanks to
Proposition \ref{P04}, this completes the proof.
\end{proof}

We next establish Theorem \ref{T2}. 
\begin{proof}[Proof of Theorem \ref{T2}]
  We use the abbreviations $\xi(t):=N(t)+1-t$ and $\zeta:=\zeta_1$, and
  start with the lower bound. We let $0<\varepsilon<1$. First note that
  there are the inclusions of events
  \begin{align*}\{\gamma(n)<\zeta\}&\supset\{\gamma(n)<\min\{\zeta,\,(1+\varepsilon)n\}\} \supset \{\xi(t)>0\textup{ for all }0\leq t\leq
                                     (1+\varepsilon)n\,,\,\gamma(n)<(1+\varepsilon)n\}\\
                                   &\supset \{\xi(t)>\rho^{1/3}t^{2/3}\textup{
                                     for all }0<t\leq
                                     (1+\varepsilon)n\,,\,\gamma(n)<(1+\varepsilon)n\}.
  \end{align*}
  In particular, with Proposition \ref{P04} at hand, we obtain for small
  $\varepsilon$ and large $n$
  \begin{align*}
\lefteqn{
    \P(X_{1, \rho}>n) = \E\left( \exp\left(-\rho\int_0^{\gamma(n)}
        (\xi(t))^{-1} \d t\right){\mathbf
        1}_{\gamma(n)<\zeta}\right)}\\
    &\geq \exp\left(-\rho\int_0^{(1+\varepsilon)n}\frac{1}{\rho^{1/3}t^{2/3}}\d t\right)\P\big(\xi(t)>\rho^{1/3}t^{2/3}\textup{
                                     for all }0< t\leq
                                     (1+\varepsilon)n\,,\,\gamma(n)<(1+\varepsilon)n\big)\\
    &\geq \exp\left(-4(\rho^2n)^{1/3}\right)\P\big(\xi(t)>\rho^{1/3}t^{2/3}\textup{
                                     for all }0<t\leq (1+\varepsilon)n\big)-\P\left(\gamma(n)\geq (1+\varepsilon)n\right).\end{align*}
From an elementary large deviation estimate for a sum of $n$ independent standard exponentials, we
know that for some $\lambda>0$
\begin{equation}
\label{E:T2-LD}
  \P\left(\gamma(n)\geq
    (1+\varepsilon)n\right)=O(\exp(-\lambda\,\varepsilon^2n)).\end{equation}
Therefore, our claim follows if we show a bound of the form
\begin{equation}
  \label{E:T2-1}
  \P\left(\xi(t)>\rho^{1/3}t^{2/3}\textup{
                                     for all }0<t\leq
      (1+\varepsilon)n\right)\geq \exp\left(-16(\rho^2n)^{1/3}\right)
  \end{equation}
  for large $n$. Essentially, this can be deduced
  from \cite[Theorem 4.1]{Por}: In the notation from there, we may
  consider the random walk $S_j = N(j)-j$,  $j\in\N$, and the function
  \[g(t):=\frac{3}{2}\rho^{1/3}(t+\max\{\rho,1\})^{2/3}-2\max\{\rho,\rho^{1/3}\}\,,\quad
    t\geq 0.\]
  The function $g$ is monotone increasing with $g(0)<0$ and regularly
  varying with index $2/3$. Moreover, it is readily checked that
\[\sup_{t\geq 1}\Big(g\big((2/3)t\big)-g\big((2/3)(t-1)\big)\Big)\leq 2/3\,.\]
Therefore, the assumptions of \cite[Theorem 4.1]{Por} are fulfilled, which
ensures after a small calculation that for $\varepsilon$ sufficiently
small and $n$ large enough, 
  \begin{equation}
\label{E:T2-2}
\P\left(S_j>g(j)\textup{
    for all }j=1,\ldots,\lfloor(1+\varepsilon)n\rfloor\right)\geq \exp\left(-16(\rho^2n)^{1/3}\right).\end{equation}
Now let us define for $0\leq t_0< t_1$ the event
\[\mathcal{E}(t_0,t_1):=\left\{\xi(t)>\rho^{1/3}t^{2/3}\textup{ for all }t_0< t\leq t_1\right\}.\]
For $j\in\N$ and $t\in\R$ with $j\leq t\leq j+1$, we have $\xi(t)\geq S_j$
and, provided $j\geq \rho_0:=8\lceil \rho\rceil$, also
\[g(j)\geq \rho^{1/3}(j+1)^{2/3}\geq \rho^{1/3}t^{2/3}.\] Therefore, by \eqref{E:T2-2}, 
\begin{equation}
\label{E:T2-3}
\P\left(\mathcal{E}(\rho_0,(1+\varepsilon)n)\right)\geq
\P\left(S_j>g(j)\textup{ for all
  }j=\rho_0,\ldots,\lfloor(1+\varepsilon)n\rfloor\right)\geq \exp\left(-16(\rho^2n)^{1/3}\right).\end{equation}
Writing
\[\P\left(\mathcal{E}(0,(1+\varepsilon)n)\right)=\P\left(\mathcal{E}(\rho_0,(1+\varepsilon)n)\,|\,\mathcal{E}(0,\rho_0)\right)\cdot
  \P\left(\mathcal{E}(0,\rho_0)\right),\] we note that
$\P\left(\mathcal{E}(0,\rho_0)\right)$ is bounded from below by a strictly
positive constant (depending on $\rho$). Moreover, since $\xi$ is a
spatially homogeneous Markov process, we clearly have
\[\P\left(\mathcal{E}(\rho_0,(1+\varepsilon)n)\,|\,\mathcal{E}(0,\rho_0)\right)\geq
  \P\left(\mathcal{E}(\rho_0,(1+\varepsilon)n)\right),\] so that our
claim \eqref{E:T2-1} follows from \eqref{E:T2-3}. En passant, let us
mention that $n^{1/3}$ is the correct stretch for the exponential
in \eqref{E:T2-1}. Indeed, this can be seen from Theorem 4.2
in \cite{Por}, where an analogous upper bound on the probability
in \eqref{E:T2-1} is given.

We now turn our attention to the upper bound. We fix a small
$0<\varepsilon<1$. On the event
$$\left\{ \gamma(n)\geq (1-\varepsilon)n\ \text{ and }\  \sup_{t\leq
    (1-\varepsilon)n}\xi(t) \leq (2\rho)^{1/3}n^{2/3}\right\},$$
we have 
$$
\exp\left( -\rho\int_0^{\gamma(n)} \xi(t)^{-1} \d t\right) \leq
\exp\left(-(1-\varepsilon)(\rho^2/2)^{1/3}n^{1/3}\right),$$ and from
Proposition \ref{P04}, $\P(X_{1, \rho}>n)$ can be bounded from above by
$$ \exp\left( -(1-\varepsilon)(\rho^2/2)^{1/3}
  n^{1/3}\right)+ \P(\gamma(n) <(1-\varepsilon)n ) + \P(\sup_{t\leq
  (1-\varepsilon)n}\xi(t) > (2\rho)^{1/3}n^{2/3}).$$
On the one hand, from an elementary large deviation estimate similar
to \eqref{E:T2-LD}, we get that for some $\lambda >0$: 
$$\P(\gamma(n) <(1-\varepsilon)n)=\P(N((1-\varepsilon)n)\geq n) = O(\exp(-\lambda\varepsilon^2n)).$$
On the other hand, $\xi$ is a L\'evy process with no negative jumps started
from $1$ such that
$$\E(\exp(q (\xi(t)-1)))=\exp\left(t(\e^q-1-q)\right), \qquad t\geq 0.$$
It follows classically that the process 
$$\exp\left( q\xi(t)-t(\e^q-1-q)\right), \qquad t\geq 0$$
is a martingale started from $\e^q$. An application of the optional
sampling theorem at the first passage time of $\xi$ above
$(2\rho)^{1/3}n^{2/3}$ yields the upper-bound
$$\exp\left(q(2\rho)^{1/3}n^{2/3} - (1-\varepsilon)n(\e^q-1-q)\right) \P(\sup_{t\leq (1-\varepsilon)n}\xi(t) > (2\rho)^{1/3}n^{2/3}) \leq \e^q.$$
Specializing this for $q=(2\rho)^{1/3}n^{-1/3}$, we deduce that for $n$
large enough
$$\P(\sup_{t\leq (1-\varepsilon)n}\xi(t) > (2\rho)^{1/3}n^{2/3})\leq \exp\left( -(1+(\varepsilon/2))(\rho^2/2)^{1/3}
  n^{1/3}\right).$$
Since $\varepsilon>0$ can be taken arbitrarily small, this completes the proof.
\end{proof}

\section{Connection with a population model with neutral mutations}
\label{sec:PopModel}
The Yule-Simon distribution originates from \cite{Sim}, where Simon
introduced a simple random algorithm to exemplify the appearance of
\eqref{E:YSd} in various statistical models. More specifically, he proposed
a probabilistic model for describing observed linguistic (but also economic
and biological) data leading to \eqref{E:YSd}. The mentioned paper
initiated a lively dispute between Simon and Mandelbrot (known as the
Simon-Mandelbrot debate) on the validity and practical relevance of Simon's
model. We mention only Mandelbrot's reply \cite{Ma} and
Simon's response \cite{Sim2}, but the discussions includes further (final)
notes and post scripta. Most interestingly, the discussion between the two
gentlemen evolved in particular around the adequacy and meaning of Simon's
model when $\rho<1$ in contrast to $\rho>1$; see pp. 95--96
in \cite{Ma}.

It is one of the purposes of this section to specify a probabilistic population
model for which the Yule-Simon law in both the cases $\rho<1$ and $\rho>1$
can be observed. More generally, we will argue that a natural
generalization of Simon's algorithm yields the two-parameter version
of \eqref{E:YSd} given in Definition \ref{D1}.  To that aim, it is
convenient to first recast Simon's model in terms of random recursive
forests, and then interpret the latter as a population model with neutral
mutations. A more general population model will then yield the full
two-parameter range of the Yule-Simon law.

\subsection{Simon's model in terms of Yule processes with mutations}
 
Fix $p\in(0,1)$, take $n \gg 1$ and view $[n]\coloneqq\{1,\ldots,n\}$ as
a set of vertices. We equip every vertex $2\leq\ell\leq n$ with a pair of
variables $(\varepsilon(\ell),u(\ell))$, independently of the other
vertices. Specifically, each $\varepsilon(\ell)$ is a Bernoulli variable
with parameter $p$, i.e.
$\P(\varepsilon(\ell)=1)=1-\P(\varepsilon(\ell)=0)=p$, and $u(\ell)$ is
independent of $\varepsilon(\ell)$ and has the uniform distribution on
$[\ell-1]$. Simon's algorithm amounts to creating an edge between $\ell $
and $u(\ell)$ if and only if $\varepsilon(\ell)=1$. The resulting random
graph is a random forest and yields a partition of $[n]$ into random
sub-trees. In this setting, Simon showed that for every $k\geq 1$, the
proportion of trees of size $k$, i.e. the ratio of the number of sub-trees
of size $k$ and the total number of sub-trees in the random forest,
converges on average as $n\to\infty$ to $\rho B(k,\rho+1)$, where
$\rho=1/p$.

Let us next enlighten the connection with a standard Yule process
$Y=Y_0$. We start by enumerating the individuals of the population model
described by the Yule process in the increasing order of their birth dates
(so the ancestor is the vertex $1$, its first child the vertex $2$, ...),
and stop the process at time 
$$T(n)\coloneqq \inf\{t\geq 0: Y(t)=n\}$$ when
the population has reached size $n$. Clearly, the parent $u(\ell)$ of an
individual $2\leq \ell \leq n$ has the uniform distribution on $[\ell-1]$,
independently of the other individuals. The genealogical tree obtained by
creating edges between parents and their children is known as a random
recursive tree of size $n$; see e.g. \cite{Drmota}. Next imagine that
neutral mutations are superposed to the genealogical structure, so that
each child is either a clone of its parent or a mutant with a new genetic
type, and more precisely, the individual $\ell $ is a mutant if and only if
$\varepsilon(\ell)=0$, where $(\varepsilon(\ell))_{\ell \geq 2}$ is a
sequence of i.i.d. Bernoulli variables with parameter $p$, independent of
the sequence $(u(\ell))_{\ell \geq 2}$. The partition of the population
into sub-populations of the same genetic type, often referred to as the
allelic partition, corresponds to an independent Bernoulli bond percolation
with parameter $p$ on the genealogical tree, that is, it amounts to
deleting each edge with probability $1-p$, independently of the other
edges. The resulting forest has the same distribution as that obtained
from Simon's algorithm.

Simon's result can then be re-interpreted by stating that the distribution
of the size of a typical sub-tree after percolation (i.e. the number of
individuals having the same genetic type as a mutant picked uniformly at
random amongst all mutants) converges as $n \to\infty$ to the Yule-Simon
distribution with parameter $\rho=1/p$. This can be established as
follows. Observe first that a typical mutant is born at time
$T(\lfloor Un\rfloor)$, where $U$ is an independent uniform variable on
$[0,1]$.  By the branching property, a typical sub-tree can thus be viewed
as the genealogical tree of a Yule process with birth rate $p$ per
individual (recall that $p$ is the probability for a child to be a clone of
its parent), stopped at time $T(n) -T(\lfloor Un\rfloor)$. Then recall that
$$\lim_{t\to \infty} \e^{-t}Y(t)=W \qquad\text{a.s.},$$ 
where $W>0$ is some random variable, and
hence
$$T(n)-T(\lfloor Un\rfloor)\sim \ln (n/W) - \ln(Un/W) = -\ln U\qquad
\text{as }n\to \infty. $$ Since a Yule process with birth rate $p$ per
individual and taken at time $t\geq 0$ has the geometric distribution with
parameter $\e^{-p t}$, and $-p\ln U$ has the exponential distribution with
parameter $\rho=1/p$, we conclude that the distribution of the size of a
typical sub-tree after percolation converges as $n \to\infty$ to
\eqref{E:YSd}.

In the following section, we shall generalize Simon's algorithm in
two different directions, leading ultimately to the two-parameter
Yule-Simon law specified in Definition \ref{D1}.

\subsection{A generalization of Simon's model}
The random algorithm described above only yields Yule-Simon distributions
with parameter $\rho>1$. A modification dealing with the case $\rho\leq 1$
has already been suggested in Simon's article, see Case II on page 431 in
\cite{Sim}; let us now elaborate on this more specifically.

\subsubsection{The full range of the one-parameter Yule-Simon law}
Rather than assuming that the $\varepsilon(\ell)$ are i.i.d. Bernoulli
variables, let us henceforth merely suppose that they form a sequence of
random variables in $\{0,1\}$, independent of the $u(\ell)$'s. As
previously, the individuals $\ell $ such that $\varepsilon(\ell)=0$ are
viewed as mutants, and those with $\varepsilon(\ell)=1$ as clones.

We write $S(n)=\sum_{j=2}^n \varepsilon(j)$ for the number of individuals
that are clones of their respective parents when the total population has
size $n$, and shall consider three mutually exclusive asymptotic regimes,
where the various limits take place in probability:
\begin{framed}
\begin{enumerate}
\item[(a)]  $\lim_{n\to \infty} S(n)/n=1/\rho$ for some $\rho>1$, 
\item[(b)] $\lim_{n\to \infty} S(n)/n=1$, and for any $r>0$,
  $\lim_{n\to \infty}(rn-S(\lfloor rn\rfloor))/(n-S(n))=r$,
\item[(c)] for any $r>0$, $\lim_{n\to \infty}(rn-S(\lfloor
rn\rfloor))/(n-S(n))=r^{\rho}$ for some $\rho \in(0,1)$.
\end{enumerate}
\end{framed}
Plainly, case (a) holds in particular when the $\varepsilon(\ell)$'s form
an i.i.d. sequence of Bernoulli variables with parameter $p=1/\rho$ as in the preceding section.
Regimes (b) and (c) are situations where mutations are asymptotically rare,
and are likely better understood in terms of the number of mutants
$\bar S(n)=n-S(n)$.  Namely (b) is equivalent to requesting that, in
probability, $\bar S(n)$ is regularly varying with index $1$ and
$\bar S(n)=o(n)$, whereas (c) requests that $\bar S(n)$ is regularly
varying with index $\rho$.

Just as before, we declare an individual $\ell$ to be a mutant if and only
if $\varepsilon(\ell)=0$, and we consider the allelic partition at time
$T(n)$, i.e., the partition of the population into sub-population bearing
the same genetic type. As we shall see in the following Proposition
\ref{P05}, this population model leads under the three different regimes to
the full range of the one-parameter Yule-Simon law when studying the limit
size of a typical sub-population.

\subsubsection{A two-parameter generalization}
It remains to appropriately extend the model in order to encompass the
two-parameter Yule-Simon distributions. To that aim, we propose a
further generalization: We replace the underlying standard Yule process
$Y=Y_0$ by a general branching process $Y_{\theta}$ as considered in the
introduction. Again, we consider independently a sequence
$(\varepsilon(\ell))_{n\geq 2}$ of $\{0,1\}$-valued random variables
indicating which individuals are clones or mutants, respectively, and
exactly as before, we may study the allelic partition at the time
\begin{equation}
\label{E:defTnTheta}
  T_{\theta}(n)=\inf\{t\geq 0: Y_{\theta}(t)=n\}\end{equation}
when the total population size $n$ is reached. We stress
that the case $\theta=0$ corresponds to the one-parameter model described
just above: We have $Y_0=Y$ and consequently $T_0(n)=T(n)$.
 
We are now in position to formulate our limit result for the proportion
of sub-populations of size $k$, generalizing Simon's result to the
two-parameter Yule-Simon distributions.
For the sake of simplicity, we focus on the case
$\theta \leq 0$ when the total population in the general branching process
$Y_{\theta}$ is infinite a.s., and leave the more delicate situation
$\theta >0$ (that requires conditioning) to interested readers. 

\begin{proposition} \label{P05} Let $\theta \leq 0$ and $\rho >0$, consider
  a general branching process $Y_{\theta}$ as in Section \ref{sec:Ytheta},
  and define $T_\theta(n)$ as in \eqref{E:defTnTheta}. Let further
  $(\varepsilon(\ell))_{n\geq 2}$ be a sequence of variables in $\{0,1\}$
  which is independent of the branching process and fulfills one of the
  regimes \textup{(a)}, \textup{(b)} or \textup{(c)}. Regard every
  individual $\ell$ with $\varepsilon(\ell)=0$ as a mutant, and consider at
  time $T_{\theta}(n)$ the (allelic) partition of the whole population into
  sub-populations of individuals with the same genetic type.

  For every $k\in\N$, write $Q_n(k)$ for the
  proportion of sub-populations of size $k$
  (i.e. the number of such sub-populations divided by the total number of
  mutants) in the allelic partition at time $T_{\theta}(n)$. 
  Then 
  $$\lim_{n\to \infty} Q_n(k) = \P(X_{\vartheta, \varrho}=k)\qquad \text{in probability},$$
 where 
$$ (\vartheta, \varrho)=   \left\{
\begin{matrix}
(\theta \rho,(1-\theta)\rho) &\text{ in regime \textup{(a)},}\\
(\theta,1-\theta) &\text{ in regime \textup{(b)},}\\
(\theta ,(1-\theta)\rho) &\text{ in regime \textup{(c)}.}\\
\end{matrix}
\right.
$$
\end{proposition}
\begin{remark} We stress that our model leads to the complete range of
  parameters $(\vartheta,\varrho)$ of the Yule-Simon distribution
  satisfying $\vartheta\leq 0$ and $\varrho>0$. Indeed, if
  $\vartheta+\varrho>1$, then the size of a typical sub-tree converges in
  law with the choices $\theta:=\vartheta/(\vartheta+\varrho)$ and
  $\rho:=\vartheta+\varrho$ under regime (a) to $X_{\vartheta,\varrho}$. If
  $\vartheta+\varrho=1$, the same choices of $\theta$ and $\rho$ lead to
  $X_{\vartheta,\varrho}$ under regime (b). If $\vartheta+\varrho<1$, then
  $\theta:=\vartheta$ and $\rho:=\varrho/(1-\vartheta)$ under regime (c)
  yield the law $X_{\vartheta,\varrho}$.
\end{remark}
  
\begin{remark}\label{R2} The conditional expectation of the size of a
  typical sub-tree given that there are $m(n)$ mutants in the population of
  total size $n$ is clearly $n/m(n)$.  Note that $m(n)\sim (1-1/\rho)n$ in
  regime (a), whereas $m(n)=o(n)$ in regimes (b) and (c).  We may thus
  expect from Proposition \ref{P05} that
$$ \E(X_{\theta \rho,(1-\theta)\rho})=\rho/(\rho-1) \qquad \text{when
  $\theta\leq 0$ and $\rho>1$},
$$
and  that 
$$ \E(X_{\theta,(1-\theta)\rho})=\infty \qquad \text{when $\theta\leq 0$ and $\rho\leq 1$}.
$$  
That these identities indeed hold has already been observed in Remark
\ref{R00}.
\end{remark} 

We shall now conclude this work by presenting the main steps of the proof
of Proposition \ref{P05}, and leaving some of the technical details to the
interested readers.
\begin{proof}[Sketch of proof of Proposition \ref{P05}]
  From Proposition \ref{P02} we deduce that the time $T_{\theta}(n)$ at
  which the population reaches size $n$ satisfies
  \[T_{\theta}(n)= (1-\theta)^{-1} \ln(n/W_{\theta})+o(1)\] in probability,
  where $W_\theta>0$ denotes the limit in probability of
  $\e^{(\theta-1)t}Y_{\theta}(t)$ as $t\rightarrow\infty$.  We first prove
  convergence in distribution for the size of a typical sub-population.

  We start with regime (a). Here, as in the case where the
  $\varepsilon(\ell)$ are i.i.d. Bernoulli variables, the probability that
  a mutant picked uniformly at random belongs to the $rn$ first individuals
  is approximately $r$ for any $r\in(0,1)$. In other words, a typical
  mutant is born at time approximately $T_{\theta}(\lfloor Un \rfloor)$,
  for $U$ an independent uniform variable on $[0,1]$. Furthermore, the size
  of the sub-tree generated by that mutant remains close to the
  genealogical tree of a general branching process where the reproduction
  point measure is Poisson with intensity $\rho^{-1} \e^{-\theta t} \d t$,
  stopped at time
  $T_{\theta}(n)-T_{\theta}(\lfloor Un \rfloor) \sim -(1-\theta)^{-1} \ln
  U$. The last expression is exponentially distributed with parameter
  $1-\theta$. It follows after the time substitution $s=\rho^{-1}t$ that
  the size of a typical sub-tree (i.e., of a typical sub-population) is
  close, in distribution, to $X_{\theta \rho,(1-\theta)\rho}$.

  As far as regime (b) is concerned, we note that the second requirement in
  regime (b) ensures that the probability that a mutant picked uniformly at
  random belongs to the $rn$ first individuals is still approximately $r$,
  for any $r\in(0,1)$.  On the other hand, the first requirement is that
  $\varepsilon(n)$ tends to $1$ in C\'esaro mean, so mutations are rare
  when $n\to \infty$. This entails that with high probability, the sub-tree
  generated by a typical mutant can be viewed as the genealogical tree of a
  general branching process with a Poisson reproduction measure of
  intensity $\e^{-\theta t} \d t$, evaluated at time
  $-(1-\theta)^{-1} \ln U$, as under (a). Therefore, its size is close, in
  distribution, to $X_{\theta,1-\theta}$.

  Finally, let us consider regime (c). Now for every $r\in(0,1)$, the
  probability that a mutant picked uniformly at random belongs to the $rn$
  first individuals is approximately $r^{\rho}$. Hence, the age of a
  typical mutant at time $T_\theta(n)$ is close in distribution to
  $-(1-\theta)^{-1}(1/\rho)\ln U$ where $U$ has again the uniform
  distribution, i.e., is close to an exponential variable with parameter
  $(1-\theta)\rho$. On the other hand, $\varepsilon(n)$ still tends to $1$
  in C\'esaro mean, so mutations are rare as $n\to \infty$, and the
  sub-tree generated by a typical mutant can again be viewed as the
  genealogical tree of a general branching process with a Poisson
  reproduction measure of intensity $\e^{-\theta t} \d t$. We conclude that
  the size of the sub-tree a typical mutant generates is close to
  $X_{\theta,(1-\theta)\rho}$.  This treats convergence in distribution for
  the size of a typical sub-population. 

  In order to pass on to the proportion $Q_n(k)$ of sub-populations
  of size $k$ at time $T_\theta(n)$, we first note that the above arguments
  readily extend to a pair of typical sub-population sizes: Indeed, if we
  choose a second mutant uniformly at random and independently of our first
  choice, then the sizes of the two respective sub-populations at time
  $T_{\theta}(n)$ become asymptotically independent as
  $n\rightarrow\infty$, in all regimes (a), (b) and (c). This implies joint
  convergence in distribution of a pair of sub-population sizes in any of
  the three regimes to a pair
  $(X_{\vartheta, \varrho},\,X'_{\vartheta, \varrho})$, where
  $X'_{\vartheta, \varrho}$ is an independent copy of
  $X_{\vartheta, \varrho}$.

  Let us now write $N_\ell(n)$ for the size of the sub-population emanating
  from individual $\ell$ at time $T_\theta(n)$ (with the convention that
  $N_\ell(n)=0$ if $\ell$ is not a mutant), and $m(n)$ for the number of
  mutants in the population at this time. For any $k\in\N$, the above
  considerations imply in particular convergence of the first two moments
  \begin{align*}
\lim_{n\rightarrow\infty}\E\left[\frac{\sum_{\ell=1}^n\1_{\{N_{\ell}(n)=k\}}}{m(n)}\right]&=\P(X_{\vartheta,
                                                                                                            \varrho}=k),\\
\lim_{n\rightarrow\infty}\E\left[\frac{\left(\sum_{\ell=1}^n\1_{\{N_{\ell}(n)=k\}}\right)^2}{m(n)^2}\right]&=\P(X_{\vartheta,
                                                                                                            \varrho}=k)^2.
  \end{align*}
  Via the second moment method, this, in turn, implies the stated
  convergence in probability for the proportion $Q_n(k)$ of sub-populations
  of size $k$.
\end{proof} 

\bibliographystyle{plain}
\bibliography{2parameterYS}
\end{document}